\documentclass[11pt,oneside]{amsart}

\title{Two polarized K3 surfaces associated to the same cubic fourfold}
\author{Emma Brakkee}
\address{Mathematisches Institut, Universit\"at Bonn, Endenicher Allee 60, 53115 Bonn, Germany}
\email{brakkee@math.uni-bonn.de}

\usepackage[marginratio={1:1, 1:1}]{geometry}
\usepackage[utf8]{inputenc}
\usepackage{enumitem}
\usepackage{etex,microtype}

\usepackage{mathtools, amsthm, amsfonts, amssymb}
\usepackage[all, cmtip]{xy}
\usepackage{centernot}

\usepackage[hidelinks]{hyperref}
\usepackage{url}
\usepackage{bookmark}

\makeatletter
\def\thmhead@plain#1#2#3{%
  \thmname{#1}\thmnumber{\@ifnotempty{#1}{ }\@upn{#2}}%
  \thmnote{ {\the\thm@notefont#3}}}
\let\thmhead\thmhead@plain
\makeatother

\newcommand\blfootnote[1]{
  \begingroup
  \renewcommand\thefootnote{}\footnote{#1}
  \addtocounter{footnote}{-1}
  \endgroup
}

\newcommand{\htab}{\hspace*{0.85em}}

\newcommand{\twostar}{$(\ast$$\ast)$} %$
\newcommand{\twostarsp}{$(\ast$$\ast)$ } %$
\newcommand{\threestar}{$(\ast$$\ast$$\ast)$} %$
\newcommand{\threestarsp}{$(\ast$$\ast$$\ast)$ } %$

\newcommand{\Z}{\mathbb{Z}}
\newcommand{\Q}{\mathbb{Q}}
\renewcommand{\P}{\mathbb{P}}
\newcommand{\C}{\mathbb{C}}
\newcommand{\R}{\mathbb{R}}
\newcommand{\N}{\mathbb{N}}

\newcommand{\mM}{\mathcal{M}}
\newcommand{\mC}{\mathcal{C}}

\newcommand{\mE}{\mathcal{E}}
\newcommand{\mCbar}{\overline{\mC}}

\newcommand{\Stau}{S^{\tau}}
\newcommand{\Ltau}{L^{\tau}}

\DeclareMathOperator{\modulo}{mod}
\renewcommand{\mod}{\:\modulo \:}

\newcommand{\lat}{\Lambda}
\DeclareMathOperator{\disc}{disc}
\DeclareMathOperator{\Disc}{Disc}

\DeclareMathOperator{\kthree}{K3}
\newcommand{\ktlat}{\Lambda_{\kthree}}
\newcommand{\polktlat}{\Lambda_d}
\newcommand{\extktlat}{\widetilde{\Lambda}_{\kthree}}

\DeclareMathOperator{\muk}{Muk}
\newcommand{\mukailat}{\Lambda_{\muk}}
\newcommand{\epart}{E_8(-1)^{\oplus 2}}

\DeclareMathOperator{\cub}{cub}
\newcommand{\cublat}{\Lambda_{\cub}}
\newcommand{\primcublat}{\Lambda_{\cub}^0}

\DeclareMathOperator{\HH}{H}
\DeclareMathOperator{\prim}{prim}
\DeclareMathOperator{\tO}{O}
\DeclareMathOperator{\id}{id}
\DeclareMathOperator{\rk}{rk}
\DeclareMathOperator{\norm}{norm}
\DeclareMathOperator{\hilb}{Hilb}
\DeclareMathOperator{\bir}{bir}
\DeclareMathOperator{\pic}{Pic}
\DeclareMathOperator{\ns}{NS}

\newcommand{\stO}{\widetilde{\tO}}
\newcommand{\kdperp}{K_d^{\perp}}
\newcommand{\vdlat}{\langle v_d\rangle}
\newcommand{\vdgroup}{\stO(\primcublat,v_d)}
\newcommand{\pmvdgroup}{\stO(\primcublat,\vdlat)}
\newcommand{\dom}{\mathcal{D}}
\newcommand{\qdom}{\mathcal{QD}}

\theoremstyle{plain}
\newtheorem{theorem}{Theorem}[section]
\newtheorem{proposition}[theorem]{Proposition}
\newtheorem{lemma}[theorem]{Lemma}
\newtheorem{corollary}[theorem]{Corollary}
\newtheorem{maintheorem}{Theorem}

\theoremstyle{definition}
\newtheorem{definition}[theorem]{Definition}
\newtheorem{remark}[theorem]{Remark}
\newtheorem{example}[theorem]{Example}

\begin{document}

\begin{abstract}
For infinitely many $d$, Hassett showed that special cubic fourfolds of discriminant $d$
are related to polarized K3 surfaces of degree $d$ via their Hodge structures.
For half of the $d$, each associated K3 surface $(S,L)$ canonically yields another one, 
$(S^{\tau},L^{\tau})$.
We prove that $S^{\tau}$ is isomorphic to the moduli space of stable coherent sheaves
on $S$ with Mukai vector $(3,L,d/6)$. We also explain for which $d$ the Hilbert schemes
$\text{Hilb}^n(S)$ and $\text{Hilb}^n(S^{\tau})$ are birational.
\end{abstract}

\maketitle

\blfootnote{The author is supported by the Bonn International Graduate School of Mathematics
and the SFB/TR 45 `Periods, Moduli Spaces and Arithmetic of Algebraic Varieties' of the DFG
(German Research Foundation).}

Special cubic fourfolds were first studied by Hassett \cite{HassettPaper}.
They are distinguished by the property that they carry additional algebraic cycles.
They arise in countably many families, parametrized by
irreducible divisors $\mC_d$ in the moduli space of cubic fourfolds.
For infinitely many $d$, the cubic fourfolds in $\mC_d$ are related to
polarized K3 surfaces of degree $d$ via their Hodge structures.
For half of the $d$, K3 surfaces associated to generic cubics in $\mC_d$ come in pairs.
The goal of this paper is to explain how two such K3 surfaces are related.

\medskip
More precisely, denote by $\mM_d$ the moduli space of polarized K3 surfaces of degree $d$.
Hassett constructed, for admissible $d$, a surjective rational map $\mM_d \dashrightarrow \mC_d$
sending a K3 surface to a cubic fourfold it is associated to.
This map is of degree two when $d\equiv 0\mod 6$ and generically injective otherwise.
In the former case, its (regular) covering involution $\tau\colon \mM_d\to\mM_d$
does not depend on the choices made to construct $\mM_d \dashrightarrow \mC_d$.
We prove the following geometric description of $\tau$.

\begin{maintheorem}[(see Thm.\ \ref{VectorPlusPol})]\label{MukaiVector}
Let $(\Stau,\Ltau)=\tau(S,L)$. Then $\Stau$ is isomorphic to the moduli space
$M_S(v)$ of stable coherent sheaves on $S$ with Mukai vector $v=(3,L,d/6)$.
\end{maintheorem}

In particular, $S$ and $\Stau$ are Fourier--Mukai partners. For general $(S,L)\in\mM_d$,
this also follows from the fact that the bounded derived categories of $S$
and $\Stau$ are both exact equivalent to the Kuznetsov category of
the image cubic fourfold \cite{AddingtonThomas}.
If $\rho(S)=1$, then $S$ is not isomorphic to $\Stau$ (as unpolarized K3 surfaces).
The number of Fourier--Mukai partners of $S$, which depends on $d$,
can be arbitrarily high \cite{Oguiso}.
The above gives a natural way of selecting one of them for each $(S,L)\in\mM_d$.

\medskip
We also explain when the Hilbert schemes of $n$ points $\hilb^n(S)$ and $\hilb^n(\Stau)$
are birational. Our main result is the following.

\begin{maintheorem}[(see Prop.~\ref{BirModelHilb}, Cor.~\ref{HilbIsoCor})]\label{HilbIso}
Suppose that $\rho(S)=1$. The following are equivalent:
\begin{enumerate}[label=\emph{(\roman*)}]
 \item $\hilb^2(S)$ and $\hilb^2(\Stau)$ are birational;
 \item $\hilb^2(S)$ and $\hilb^2(\Stau)$ are isomorphic;
 \item There exists an integral solution to the equation $3p^2-(d/6)q^2=-1$.
\end{enumerate}
\end{maintheorem}

We will see that this condition is satisfied for infinitely many $d$ but not for all of them.
As an application, we obtain an example of derived equivalent Hilbert schemes of two points
on K3 surfaces which are not birational.

\subsection*{Acknowledgements}
This work is part of my research as a PhD candidate.
I am grateful to my advisor Daniel Huybrechts for many helpful suggestions and comments.
I would also like to thank Giovanni Mongardi and Andrey Soldatenkov for their help,
and Pablo Magni for comments on the first version.

\section{Lattices}
In this section we set up the notation for the lattice theory that will be needed, see
\cite[Ch.~14]{LecturesOnK3} for references.

\medskip
For a lattice $\lat$ with intersection form $(\;,\;)\colon \lat\times \lat\to\Z$,
we denote by $\lat^{\vee}$ its dual lattice
and by $\Disc\lat=\lat^{\vee}/\lat$ its discriminant group.
This is a finite group of order $|\disc\lat|$,
the absolute value of the discriminant of $\lat$.
Every orthogonal transformation $g\in\tO(\lat)$ of $\lat$ induces an automorphism
on $\Disc\lat$, which we denote by $\overline{g}$.
For a primitive sublattice $\lat_1$ of a unimodular lattice $\lat$
with orthogonal complement $\lat_2=\lat_1^{\perp}$,
there exists a natural isomorphism $\Disc\lat_1\cong\Disc\lat_2$.
We will use the following fact\footnote{
A proof in the case of even lattices (for a slightly weaker statement) can be found in
\cite[Prop.~14.2.6]{LecturesOnK3}. The general statement is proven similarly.
}:

\begin{lemma}\label{extendauto}
If $g_1\in\tO(\lat_1)$ and $g_2\in\tO(\lat_2)$, then
${g_1\oplus g_2\colon \lat_1\oplus \lat_2\to \lat_1\oplus \lat_2}$
extends to an orthogonal transformation of $\lat$ if and only if $\overline{g}_1=\overline{g}_2$
under the identification ${\Disc\lat_1\cong\Disc\lat_2}$.
\end{lemma}

The first type of lattices that we use comes from K3 surfaces.
The middle cohomology $\HH^2(S,\Z)$ of a K3 surface $S$ (with the usual intersection pairing)
is isomorphic to the \emph{K3 lattice}
\[\ktlat :=\epart\oplus U^{\oplus 3}=\epart\oplus U_1\oplus U_2\oplus U_3.\]
We denote the standard basis of $U_i$ by $e_i,f_i$.
On the full cohomology $\HH^*(S,\Z)$ of $S$ we consider the \emph{Mukai pairing}, given by
$\bigl((x_0,x_2,x_4),(x'_0,x'_2,x'_4)\bigr)=x_2x'_2-x_0x'_4-x'_0x_4$
for $x_i,x'_i\in\HH^i(S,\Z)$. With this pairing, $\HH^*(S,\Z)$ becomes isomorphic to
the \emph{Mukai lattice}
\[\mukailat:=\ktlat\oplus U(-1)=\epart\oplus U_1\oplus U_2\oplus U_3\oplus U_4(-1).\]
As $U\cong U(-1)$, the Mukai lattice is isomorphic to $\ktlat\oplus U$.
To avoid confusion, we denote the latter by $\extktlat$, and fix an isomorphism
$\extktlat\xrightarrow{\sim}\mukailat$ by sending $f_4$ to $-f_4$.

\medskip
We fix $\ell_d = e_3+\tfrac{d}2f_3\in U_3\subset \ktlat$ and
let $\polktlat:=\ell_d^{\perp}\subset\ktlat$ be its orthogonal complement in $\ktlat$. Then
\[\polktlat\cong \epart\oplus U^{\oplus 2}\oplus\Z(-d)\]
is isomorphic to the primitive cohomology $L^{\perp}\subset \HH^2(S,\Z)$
of any polarized K3 surface $(S,L)$ of degree $d$.
We will need the following subgroup
\[\stO(\polktlat):= \{f\in \tO(\polktlat) \mid \overline{f}=\id_{\Disc\polktlat}\}\]
of $\tO(\polktlat)$,
which, by Lemma \ref{extendauto}, is isomorphic to
$\{f\in \tO(\ktlat) \mid f(\ell_d)=\ell_d\}$.

\medskip
Next, we define some lattices related to cubic fourfolds. Fix a primitive embedding of the lattice
$A_2=\langle \lambda_1,\lambda_2\rangle
=\left( \Z^{\oplus 2},\left(\begin{smallmatrix} 2 & -1 \\ -1 & 2 \end{smallmatrix}\right)\right)$
into $U_3\oplus U_4\subset\extktlat$,
for instance by $\lambda_1\mapsto e_3+f_3$ and $\lambda_2\mapsto e_4+f_4-e_3$.
This embedding is unique up to composition with elements of $\tO(\extktlat)$.
We are mostly interested in the complement $A_2^{\perp}\subset\extktlat$ of $A_2$:
\[A_2^{\perp}\cong \epart\oplus U_1\oplus U_2\oplus A_2(-1).\]

Denote by $\HH^4(X,\Z)^-$ the middle cohomology of a cubic fourfold $X$,
with the intersection product changed by a sign.
This lattice is isomorphic to
\[\cublat:=\epart\oplus U^{\oplus 2}\oplus \Z(-1)^{\oplus 3}.\]
Let $h=(1,1,1)\in \Z(-1)^{\oplus 3}\subset \cublat$.
The primitive cohomology
$\HH^4(X,\Z)^-_{\prim}\subset\HH^4(X,\Z)^-$ of $X$
is isomorphic to $\primcublat:=h^{\perp}\subset \cublat$. An easy computation shows that
\[\primcublat\cong \epart\oplus U^{\oplus 2}\oplus A_2(-1),\]
so $\primcublat\cong A_2^{\perp}\subset \extktlat$.
As for $\polktlat$, we will consider the subgroup
\[\stO(\primcublat):=\{f\in \tO(\primcublat) \mid \overline{f}=\id_{\Disc\primcublat}\}
\cong \{f\in \tO(\cublat) \mid f(h)=h\}\]
of $\tO(\primcublat)$ acting on $\primcublat$.

\section{Hassett's construction}
We summarize Hassett's construction, explaining those proofs that we need for our results.
For details, see \cite{HassettPaper}.

\subsection{Special cubic fourfolds}
As above, we denote by $\HH^4(X,\Z)^-$ the middle cohomology lattice
of a cubic fourfold $X$ with the intersection form changed by a sign.
Inside it, we consider the lattice $A(X)$ of Hodge classes:
\[A(X)=\HH^4(X,\Z)^-\cap \HH^{2,2}(X)\]
which is negative definite by the Hodge--Riemann bilinear relations.
We also fix the notation $h_X\in\HH^4(X,\Z)$ for the square of a hyperplane class on $X$.
For $X$ general, the lattice $A(X)$ has rank one and is generated by $h_X$.
We call $X$ \emph{special} if $\rk A(X)\geq 2$.
By the Hodge conjecture for cubic fourfolds \cite{ZuckerHC}, $X$ is special if and only if
$X$ contains a surface that is not homologous to a complete intersection.

\medskip
If $X$ is special, then $A(X)$ contains a primitive sublattice $K$ of rank two. Hassett proved
that fixing the discriminant of such $K$ gives divisors in the moduli space $\mC$
of smooth cubic fourfolds. Namely, define
\[\mC_d:=\{X\in\mC\mid \exists K\subset A(X),\; h_X\in K,\; \rk K=2,\; \disc K = d\}.\]
Then the set of special cubic fourfolds in $\mC$ is the union of all $\mC_d$.
\begin{theorem}\textnormal{\cite[Thm.~1.0.1]{HassettPaper}}
 The set $\mC_d$ is either empty or an irreducible divisor in $\mC$.
 It is non-empty if and only if $d$ satisfies
 \[(\ast)\htab d>6 \text{ and } d\equiv 0,2 \text{ mod }6.\]
\end{theorem}

\subsection{Periods of special cubic fourfolds}
Recall the definition of the period domain for a lattice
$\lat$ of signature $(n_+,n_-)$ with $n_+\geq 2$:
\[\dom(\lat)=\{x\in\P(\lat\otimes\C)\mid (x)^2=0,\: (x,\bar{x})>0\}.\]
This is a complex manifold of dimension $\rk\lat -2$, which is connected when $n_+>2$
and has two connected components when $n_+=2$. In the second case,
the components are interchanged by complex conjugation.

\medskip
The period domain for polarized K3 surfaces of degree $d$ is $\dom(\polktlat)$,
which is a 19-dimensional manifold with two connected components.
By Baily--Borel \cite{BB}, its quotient ${\qdom(\polktlat):=\stO(\polktlat)\backslash\dom(\polktlat)}$
by the group $\stO(\polktlat)$ is a connected, quasi-projective variety of dimension 19.
Furthermore, by the global Torelli theorem for K3 surfaces \cite{PjaSha},
the period map induces an open embedding
$\mM_d\hookrightarrow \qdom(\polktlat)$
(this map is algebraic by Borel's extension theorem \cite{Borel},
see e.g.\ \cite[Rem.~6.4.2]{LecturesOnK3}).

\medskip
The period domain for cubic fourfolds is $\dom(\primcublat)$,
a 20-dimensional manifold with again two connected components.
By Baily--Borel, the quotient
$\qdom(\primcublat):= \stO(\primcublat)\backslash \dom(\primcublat)$
is a connected quasi-projective variety of dimension 20.
By the global Torelli theorem for cubic fourfolds \cite{VoisinTorelli},
the period map gives an open embedding
$\mC\hookrightarrow \qdom(\primcublat)$
(algebraicity again follows from \cite{Borel}, see \cite[Prop.~2.2.2]{HassettPaper}).

\medskip
Inside $\dom(\primcublat)$ we can identify those periods coming from special cubic fourfolds.
Note that a cubic fourfold $X$ is special if and only if there exists
a negative definite sublattice ${K\subset \HH^4(X,\Z)^-}$ of rank two with $h_X\in K$, such that
$K\otimes\C$ is contained in ${\HH^{3,1}(X)^{\perp}\subset\HH^4(X,\C)^-}$.
On the level of the period domain, this means the following:
After choosing a marking
$\HH^4(X,\Z)^-_{\prim}\xrightarrow{\sim}\primcublat$,
the period of $X$ lands in
\[\{x\in \dom(\primcublat)\mid (K\cap \primcublat)_{\C}\subset x^{\perp}\}\]
for some primitive, negative definite sublattice $K\subset \cublat$ of rank two containing $h$.
Let $K^{\perp}\subset \primcublat$ be its orthogonal complement,
then the set above is equal to the divisor ${\dom(K^{\perp})\subset\dom(\primcublat)}$.

\medskip
We fix one sublattice $K_d\subset\cublat$ as above, with discriminant $d$.
Let $\mCbar_d\subset \qdom(\primcublat)$ be the image of $\dom(\kdperp)\subset\dom(\primcublat)$
under the quotient map $\dom(\primcublat)\to \qdom(\primcublat)$.
The following shows that $\mCbar_d$ does not depend on the choice of $K_d$.

\begin{proposition}[{\cite[Prop.~3.2.4]{HassettPaper}}]
 Let $K,K'\subset \cublat$ be primitive sublattices of rank two containing $h$.
 Then $K=f(K')$ for some $f\in \stO(\primcublat)$ if and only if $\disc K=\disc K'$.
\end{proposition}

Note that the immersion $\mC\hookrightarrow\qdom(\primcublat)$
maps $\mC_d$ into $\mCbar_d$. In fact, we have $\mC_d=\mC\cap\mCbar_d$.

\subsection{Associated K3 surfaces}
Consider the following condition on $d\in\N$:
\[\text{\twostarsp}\htab \text{$d$ is even and not divisible by 4, 9,
or any odd prime $p\equiv 2$ mod 3}.\]
This implies that $d\equiv 0,2\mod 6$.
Hassett proved the following statement:

\begin{proposition}[{\cite[Prop.~5.1.4]{HassettPaper}}]
 The number $d$ satisfies \twostarsp  if and only if there is an isomorphism
$\polktlat\cong\kdperp$.
\end{proposition}

So when $d$ satisfies \twostar, there is an isomorphism of period domains
$\dom(\polktlat)\cong\dom(\kdperp)$.
Under the identification $\polktlat\cong\kdperp$, the group
$\stO(\polktlat)$ forms a subgroup of $\stO(\primcublat)$,
see Proposition~\ref{identifyOrth} below, so we also get a surjective map
$\qdom(\polktlat) = \tO(\polktlat)\backslash\dom(\polktlat)\to \mCbar_d$.
This gives us the following commutative diagram:
  \[\xymatrix{\dom(\polktlat) \ar@{->>}[d] \ar[r]^{\sim} & \dom(\kdperp)\, \ar@{->>}[d] \ar@{^{(}->}[r]
    & \dom(\primcublat) \ar@{->>}[d] \\
  \qdom(\polktlat) \ar@{->>}[r] & \mCbar_d \, \ar@{^{(}->}[r] & \qdom(\primcublat) \\
  \mM_d \ar@{}[u]^(.18){}="a"^(.88){}="b" \ar@{^{(}->} "a";"b" \ar@{-->}[r]^{\varphi}
    & \mC_d \, \ar@{^{(}->}[r] \ar@{}[u]^(.18){}="a"^(.88){}="b" \ar@{^{(}->} "a";"b"
    & \mC \ar@{}[u]^(.18){}="a"^(.88){}="b" \ar@{^{(}->} "a";"b"
  }\]

It can be shown that the rational map $\varphi\colon\mM_d\dashrightarrow \mC_d$
is regular on an open subset which maps surjectively to $\mC_d$,
see \cite[p.~14]{HassettPaper}.
Note that $\varphi$ depends on the choice of an isomorphism $\polktlat\cong\kdperp$,
thus it is only unique up to $\tO(\polktlat)/\stO(\polktlat)$.

\medskip
If $\varphi$ sends $(S,L)\in\mM_d$ to $X$
then there exists, up to a Tate twist, an isometry of Hodge structures
\[\HH^4(X,\Z)^-\supset K^{\perp}\cong L^{\perp}\subset\HH^2(S,\Z)\]
for some primitive sublattice $K\subset A(X)$ of rank two and discriminant $d$ containing $h_X$.
Conversely, if such a Hodge isometry exists, this induces a lattice isomorphism
\[\polktlat\cong L^{\perp}\cong K^{\perp}\cong\kdperp\]
such that the induced map $\varphi\colon\mM_d\dashrightarrow \mC_d$
sends $(S,L)$ to $X$. This motivates the following definition.

\begin{definition}
Let $X\in\mC_d$.
A polarized K3 surface $(S,L)\in\mM_d$ is \emph{associated} to $X$ if there exists a
Hodge isometry
\[\HH^4(X,\Z)^-(1)\supset K^{\perp}\cong L^{\perp}\subset\HH^2(S,\Z)\]
for some primitive sublattice $K\subset A(X)$ of rank two and discriminant $d$ containing $h_X$.
\end{definition}
For the rest of this section, we fix one choice of the rational map $\varphi$.
We suppress the Tate twist and view $\HH^4(X,\Z)^-$ as a Hodge structure of weight two.

\begin{remark}
The complement of the image of the inclusion $\mC\hookrightarrow\qdom(\primcublat)$ is
$\mCbar_2\cup\mCbar_6$, see \cite{Laza,Looij}.
Therefore, $\varphi$ is defined on $(S,L)\in\mM_d$
if and only if its image under ${\qdom(\polktlat)\to\mCbar_d}$
is contained in $\mCbar_d\backslash \left(\mCbar_2\cup\mCbar_6\right)$.
In particular, this holds when $\rho(S)=1$.
\end{remark}

To describe the map $\qdom(\polktlat)\to \mCbar_d$, we define two subgroups of $\stO(\primcublat)$.
First, consider the group of isomorphisms of $\cublat$ that are the identity on $K_d$.
Elements of $\stO(\primcublat)$ are in this group if and only if they fix a generator
$v_d$ of $K_d\cap \primcublat$, which is unique up to a sign.
We denote the group by $\vdgroup$:
\begin{align*}
 \vdgroup&=\{f\in\tO(\cublat)\mid f|_{K_d}=\id_{K_d}\}\\
 &=\{f\in\stO(\primcublat)\mid f(v_d)=v_d\}.
\end{align*}

The next statement is part of \cite[Thm.~5.2.2]{HassettPaper}.
It follows directly from Lemma~\ref{extendauto}.
\begin{proposition}\label{identifyOrth}
Suppose that $d$ satisfies \twostar. Under the isomorphism
$\polktlat\cong\kdperp$, the group $\stO(\polktlat)$ is identified with $\vdgroup$.
\end{proposition}
In particular, there is an isomorphism
$\qdom(\polktlat)\cong \vdgroup \backslash \dom(\kdperp)$.

\medskip

The second group we consider consists of the elements $f\in\stO(\primcublat)$ that preserve $K_d$.
These are exactly the ones satisfying $f(v_d)\in\{v_d,-v_d\}$.
We therefore denote this group by $\pmvdgroup$:
\begin{align*}
\pmvdgroup&=\{f\in\tO(\cublat)\mid f(h)=h\text{ and }f(K_d)=K_d\}\\
&= \{f\in \stO(\primcublat)\mid f(v_d) = \pm v_d\}.
\end{align*}

Again by Baily--Borel, the quotient $\pmvdgroup\backslash\dom(\kdperp)$ is a normal
quasi-projective variety. In fact, the map
$\pmvdgroup\backslash\dom(\kdperp)\twoheadrightarrow\mCbar_d$
is the normalization of $\mCbar_d$.\footnote{
A generic special cubic fourfold $X$ satisfies $\rk A(X)=2$ \cite[Sec.~5.1]{Zarhin},
so there is only one sublattice $K_d\subset A(X)$. It follows that the map
$\pmvdgroup\backslash\dom(\kdperp)\twoheadrightarrow\mCbar_d$
is generically injective.
To see that it is proper, note that the action of $\stO(\primcublat)$ on $\dom(\primcublat)$ is
properly discontinuous \cite[Rem.~6.1.10]{LecturesOnK3}.
Hence the map $\dom(\primcublat)\to \qdom(\primcublat)$
is closed, as is its restriction $\dom(\kdperp)\twoheadrightarrow\mCbar_d$ to the
closed subset $\dom(\kdperp)\subset \dom(\primcublat)$.
Since this factors as
\[\dom(\kdperp)\twoheadrightarrow \pmvdgroup\backslash\dom(\kdperp)\twoheadrightarrow\mCbar_d,\]
the map $\pmvdgroup\backslash\dom(\kdperp)\twoheadrightarrow\mCbar_d $
is closed as well.
Moreover, it has finite fibres, so it is proper.
}
Summarizing, we have the following commutative diagram:
  \[\xymatrix@C=0em{\qdom(\polktlat) & \cong \vdgroup\backslash\dom(\kdperp)
    \overset{\overline{\gamma}}{\longrightarrow}\pmvdgroup\backslash\dom(\kdperp)\cong
    & \mCbar_d^{\norm}\ar@{->>}[rrr] & & & \mCbar_d \\
  \mM_d\ar@{}[u]^(.18){}="a"^(.88){}="b" \ar@{^{(}->} "a";"b"
     \ar@{-->}[rr]^{\gamma} & & \mC_d^{\norm}
    \ar@{->>}[rrr] & & & \mC_d \ar@{}[u]^(.18){}="a"^(.88){}="b" \ar@{^{(}->} "a";"b"
  }\]

The spaces $\vdgroup\backslash\dom(\kdperp)$ and
$\pmvdgroup\backslash\dom(\kdperp)$
can be seen as period domains of \emph{marked} and \emph{labelled} cubic fourfolds,
respectively, see \cite[Sec.~3.1,~5.2]{HassettPaper}.

\medskip
The following describes the generic fibre of the quotient map
\[\overline{\gamma}\colon\vdgroup\backslash\dom(\kdperp)\twoheadrightarrow
\pmvdgroup\backslash\dom(\kdperp).\]

\begin{proposition}[{\cite[Prop.~5.2.1]{HassettPaper}}]\label{indextwo}
There is an isomorphism
\[\pmvdgroup/ \vdgroup\cong
\begin{cases}
    \{0\}       &  \text{if } d\equiv 2\mod 6\\
    \Z/2\Z  &  \text{if } d\equiv 0\mod 6.
  \end{cases}\]
\end{proposition}

As a consequence, $\overline{\gamma}$ is an isomorphism when $d\equiv 2 \mod 6$
and has degree two when $d\equiv 0\mod 6$.
In the latter case, the covering involution of $\overline{\gamma}$ is induced
by an automorphism $g\in\pmvdgroup$ whose class modulo $\vdgroup$ generates $\Z/2\Z$.
In the proof of the proposition for $d\equiv 0\mod 6$, which we will explain below,
an explicit such $g$ is constructed.

\medskip
Fix a primitive negative definite sublattice $K_d\subset \cublat$ containing $h$
of rank two and discriminant $d$, where $3|d$.
Let $v_d$ be a generator of $K_d\cap \primcublat$.
Suppose $T\in K_d$ is a primitive element such that $h$ and $T$ generate $K_d$.
Then $(h,T)$ is also divisible by 3. We can write
$v_d=\tfrac13(h,T)h-T$, which has square $-d/3$.
It follows that $K_d=\Z h\oplus\Z v_d$.

\begin{lemma}\label{congruence}
 Let $x\in \primcublat$ be primitive with $(x,x)\neq 0$ and $3\centernot|(x,x)$.
There exists an $f\in\stO(\primcublat)$ such that $f(x)=e_2+\tfrac{(x,x)}2 f_2$.
\end{lemma}
\begin{proof}
As $\Disc\primcublat = \Z/3\Z$, the integer
$n$ defined by $\{(x,y)\mid y\in\primcublat\}=n\Z$
is either 1 or 3. Since $3\centernot|(x,x)$, it must be 1.
Therefore the class $\overline{x}\in\Disc\primcublat$ is trivial.
By Eichler's criterion \cite[Prop.~3.3]{GrHuSa}
there is an $f\in\stO(\primcublat)$ sending $x$ to $e_2+\tfrac{(x,x)}2 f_2$.
\end{proof}

\begin{proof}[Proof of Proposition~\ref{indextwo}.]
Any $g\in\pmvdgroup$ whose class modulo $\vdgroup$ is non-zero satisfies $g(v_d)=-v_d$.
We construct an explicit such $g$.
By Lemma~\ref{congruence}, we can assume
\[v_d=e_2-\tfrac{d}6 f_2\in\cublat = E_8(-1)^{\oplus 2}\oplus U_1\oplus U_2\oplus \Z(-1)^{\oplus 3}.\]
We extend $v_d\mapsto -v_d$ to an element of $\stO(\primcublat)$
by taking multiplication with $-1$ on $U_1\oplus U_2$
and the identity on the rest of $\cublat$.
This is unique up to composition with elements of $\vdgroup$.
\end{proof}

We will denote by $g$ the automorphism of $\cublat$ given by
\[g|_{E_8(-1)^{\oplus 2}\oplus \Z(-1)^{\oplus 3}}=\id;\; g|_{U_1\oplus U_2}=-\id.\]

\section{The involution on \texorpdfstring{$\mM_d$}{Md}}\label{InvoSection}

We will now prove Theorem \ref{MukaiVector}.
In the previous section, we explained that the map
\[\overline{\gamma}\colon \qdom(\polktlat)\twoheadrightarrow \pmvdgroup\backslash\dom(\kdperp)\]
is an isomorphism if $d\equiv 2 \mod 6$
and has degree two if $d\equiv 0\mod 6$.
In the second case, define $\tau\colon\qdom(\polktlat)\to\qdom(\polktlat)$
to be the covering involution of $\overline{\gamma}$.
Note that $\tau$ maps $\mM_d$ to itself:
As explained in e.g.\ \cite[Rem.~6.3.7]{LecturesOnK3},
the complement of $\mM_d$ in $\qdom(\polktlat)$ is
\[\bigcup_{\delta\in\polktlat,\: \delta^2=-2}\delta^{\perp}\]
and this set is clearly preserved under $g$.

\begin{proposition}\label{TauIndependent}
The morphism $\tau$ does not depend on the choice of $\polktlat\cong\kdperp$.
\end{proposition}
\begin{proof}
Precomposing the isomorphism $\polktlat\cong\kdperp$ with $f\in\tO(\polktlat)$
changes $g$ on $\polktlat$ to ${f^{-1}\circ g\circ f}$.
Note that this has the same action on $\Disc(\polktlat)\cong\Z/d\Z$ as $g$,
thus it induces the same action on
$\qdom(\polktlat)=\stO(\polktlat)\backslash\dom(\polktlat)$.
\end{proof}

For a K3 surface $S$, we denote by $\widetilde{\HH}(S,\Z)$
the full cohomology of $S$ with the Mukai pairing
and the Hodge structure of weight two defined by $\widetilde{\HH}{}^{2,0}(S):=\HH^{2,0}(S)$.
For a primitive vector $v=(r,\ell,s)\in\widetilde{\HH}(S,\Z)$,
denote by $M_S(v)$ the moduli space of stable coherent sheaves on $S$ with Mukai vector $v$,
with respect to a generic polarization.
Recall \cite[Ch.~10]{LecturesOnK3} that if there exists a $w$ in
$\widetilde{\HH}{}^{1,1}(S,\Z)= \HH^0(S,\Z)\oplus\HH^{1,1}(S,\Z)\oplus\HH^4(S,\Z)$
with $(v,w)=1$, then this is a fine moduli space.
If $(v)^2=0$ and $r>0$, then $M_S(v)$ is a K3 surface.

\begin{theorem}\label{VectorPlusPol}
Let $(\Stau,\Ltau)=\tau(S,L)$. Then $\Stau$ is isomorphic to the moduli space $M_S(v)$
of stable coherent sheaves on $S$ with Mukai vector $v=(3,L,d/6)$.
Under the natural identification
$\HH^2(\Stau,\Z)\cong v^{\perp}/\Z v\subset\widetilde{\HH}(S,\Z)$,
we have $\Ltau=(d,(d/3-1)L,d/3(d/6-1))$.
\end{theorem}

Let us describe the strategy of the proof.
The restriction of $g\in\tO(\cublat)$ to $\kdperp$ can be viewed as an involution on $\polktlat$.
Because $g$ does not induce the identity on $\Disc\kdperp$,
this involution does not extend to an automorphism on $\ktlat$.
However, we will show that $g$ extends to $\widetilde{g}\in\tO(\extktlat)$.

\medskip
Using $\widetilde{g}$, we find $\Stau$ as follows.
Let $x$ be a representative of the period of $(S,L)$ in
${\Lambda_{d,\C}\subset \widetilde{\Lambda}_{\kthree,\C}}$.
The K3 surface $\Stau$ is the one whose period can be represented by $g(x)\in\Lambda_{d,\C}$.
The map $\widetilde{g}$ induces a Hodge isometry
$\widetilde{\HH}(S,\Z)\cong\widetilde{\HH}(\Stau,\Z)$,
so by the derived Torelli theorem (\cite{Orlov}, see also \cite[Prop.~16.3.5]{LecturesOnK3}),
$S$ and $\Stau$ are Fourier--Mukai partners.

More precisely, denote by $\widetilde{g}_{\muk}$ the morphism $\widetilde{g}$,
seen as an automorphism of $\mukailat$.
Let ${v=(r,\ell,s):=\widetilde{g}_{\muk}^{-1}(0,0,1)}$.
We will see that $r>0$, so there exists a universal sheaf $\mE$ on $S\times M_S(v)$.
Let $\Phi^H_{\mE}\colon \widetilde{\HH}(M_S(v),\Z)\to \widetilde{\HH}(S,\Z)$
be the induced cohomological Fourier--Mukai transform.
Then $(\Phi^H_{\mE})^{-1}\circ \widetilde{g}_{\muk}^{-1}$ sends
$\HH^2(\Stau,\Z)$ to $\HH^2(M_S(v),\Z)$, which shows that $\Stau$ is isomorphic to $M_S(v)$.

\medskip
To describe $\Ltau$, note that $\Phi^H_{\mE}$ induces an isomorphism
$\HH^2(M_S(v),\Z)\cong v^{\perp}/\Z v$, where $v^{\perp}\subset \widetilde{\HH}(S,\Z)$
(this is a result by Mukai, see \cite[Rem.~10.3.7]{LecturesOnK3}).
Thus, $\widetilde{g}_{\muk}^{-1}$ restricts to an isomorphism $\HH^2(\Stau,\Z)\cong v^{\perp}/\Z v$.
Under this identification, the polarization $\Ltau$ is mapped to
$\widetilde{g}_{\muk}^{-1}(\ell_d)$.

\begin{remark}
The extension $\widetilde{g}$ of $g$ is not unique. But if $\widetilde{g}'$ is another extension,
then $\widetilde{g}_{\muk}^{-1}\circ \widetilde{g}_{\muk}'$ is an orthogonal transformation
of $\mukailat$ sending $v' = (\widetilde{g}'_{\muk})^{-1}(0,0,1)$ to $v$.
This induces a Hodge isometry $\HH^2(M_S(v'),\Z)\cong \HH^2(M_S(v),\Z)$,
so $M_S(v')$ and $M_S(v)$ are isomorphic.
\end{remark}

\begin{remark}
The space $\qdom(\polktlat)$ can be interpreted as the moduli space of \emph{quasi-polarized}
K3 surfaces, i.e.\ pairs $(S,L)$ with $L$ the class of a big and nef line bundle,
see \cite[Sec.~5]{HulekPloog}. For such pairs the theorem is still valid.
\end{remark}

\begin{remark}
For $d\equiv 0\mod 6$, the ramification locus of $\overline{\gamma}$ over $\mM_d$ consists of those $(S,L)$
which are polarized isomorphic to $(\Stau,\Ltau)$.
It follows from \cite[Sec.~8]{HulekPloog} that $\overline{\gamma}$ is unramified over
$\{(S,L)\in\mM_d\mid \rho(S)=1\}$.
\end{remark}

\subsection{Proof of Theorem \protect\ref{MukaiVector}}
We will first compute the action of $g$ on $\Disc\kdperp \cong\Disc\polktlat$.
We have seen that $K_d\cong\Z h \oplus\Z v_d$, so
\[\Disc\kdperp\cong\Disc K_d\cong\Z/3\Z\oplus \Z/\tfrac{d}3\Z\cong\Z/d\Z,\]
since $9\centernot| d$.

\begin{lemma}
The action of $g$ on $\Disc \kdperp \cong\Z/d\Z$ is given by $x\mapsto(d/3-1)x$.
\end{lemma}
\begin{proof}
On $\Z/3\Z\oplus \Z/\tfrac{d}3\Z$, the map $\overline{g}$ is given by
$(1,0)\mapsto (1,0)$ and $(0,1)\mapsto(0,-1)$ modulo $d$.
Let $\alpha$ be such that $\overline{g}$ acts on $\Z/d\Z$ by $x\mapsto \alpha x$.
We can compute $\alpha$ as follows: let $(s,t)$ be any generator of $\Z/3\Z\oplus \Z/\tfrac{d}3\Z$.
Then
\[((\alpha+1)s,(\alpha+1)t) = \overline{g}(s,t)+(s,t) = (2s,0)\]
has order 3.
Since $9\centernot| d$, this means that $t(\alpha+1)\equiv 0\mod d/3$, and thus
$\alpha\equiv -1\mod d/3$.
So $\alpha$ is in $\{-1,d/3-1,2d/3-1\}$.

We directly see that $\alpha = -1$ is not possible, since then $\alpha(s,t)+(s,t)=0 \neq 2s$.
Next, we should have $q_{\kdperp}(\overline{g}(s,t)) = q_{\kdperp}(s,t)\in\Q/2\Z$.
Now $q_{\kdperp}(s,t) = n/d$ for some $n\in\Z$ with $\gcd(n,d) = 1$.
Suppose that $\alpha=2d/3-1$. Then
\[q_{\kdperp}(\alpha(s,t)) = (2d/3-1)^2q_{\kdperp}(s,t) = 4n/3(d/3-1)+q_{\kdperp}(s,t).\]
Since $n$ is not divisible by 3 and $d/3\equiv 2\mod 3$,
the number $4n/3(d/3-1)$ is not an integer,
so $q_{\kdperp}(\alpha(s,t))\neq q_{\kdperp}(s,t)$.
We conclude that $\alpha = d/3-1$.
\end{proof}

To extend $g$ on $\polktlat$ to $\extktlat$ it thus suffices, by Lemma~\ref{extendauto},
to find an orthogonal transformation of $\polktlat^{\perp}=\Z \ell_d\oplus U_4$
acting on the discriminant group by $x\mapsto(d/3-1)x$.
Consider $u\in\tO(\Z \ell_d\oplus U_4)$ defined by
\begin{align*}
 e_4&\mapsto -\tfrac{d}6 e_4-\tfrac13 \left(\tfrac{d}6-1\right) \ell_d+
 \tfrac13\left(\tfrac{d}6-1\right)^2f_4\\
 f_4&\mapsto 3e_4+\ell_d-\tfrac{d}6f_4\\
 \ell_d&\mapsto de_4+\left( \tfrac{d}3-1\right)\ell_d-\tfrac{d}3\left(\tfrac{d}6-1\right)f_4.
\end{align*}
One computes that this is an involution.
The discriminant group $\Disc(\Z \ell_d\oplus U_4)\cong\Z/d\Z$
is generated by the class $\overline{\tfrac{1}{d}\ell_d}$,
which $\overline{u}$ multiplies by $d/3-1$.
So the action of $\overline{u}$ on $\Z/d\Z$ is given by $x\mapsto (d/3-1)x$.

\medskip
It follows that
$g\oplus u\in\tO(\polktlat\oplus \Z\ell_d\oplus U_4)$
extends to $\widetilde{g}\in\tO(\extktlat)$.
Since $\widetilde{g}$ is an involution, we have
$\widetilde{g}^{-1}(f_4) = \widetilde{g}(f_4) = 3e_4+\ell_d-\tfrac{d}6f_4$.
As an element of the Mukai lattice, this is
\[v=(3,\ell_d,d/6)\in \mukailat = \Z e_4\oplus\ktlat\oplus \Z(-f_4).\]
The polarization $\Ltau=\widetilde{g}^{-1}(\ell_d)$,
seen as an element of $v^{\perp}/\Z v\subset \widetilde{\HH}(S,\Z)$, is
\[\Ltau = (d,(d/3-1)L,d/3(d/6-1)).\]
This finishes the proof of Theorem~\ref{MukaiVector}.

\begin{remark}\label{otherd}
As shown in the proof of Proposition \ref{TauIndependent}, we can replace $g$ by any
element of $\tO(\polktlat)$ with the same action on $\Disc\polktlat$.
For instance, we can take the automorphism given by the identity
on $\epart\oplus U_1$ and $u$ on $U_2\oplus \Z(e_3-\tfrac{d}2f_3)\cong (\Z\ell_d\oplus U_4)(-1)$.
This allows us to define $\tau$ on $\mM_d$ for all $d\equiv 0\mod 6$ with $d/6\equiv 1\mod 3$.
\end{remark}

\section{Birationality of Hilbert schemes}
In this section we study the Hilbert schemes of $n$ points $\hilb^n(S)$ and $\hilb^n(\Stau)$
of our K3 surfaces $S$ and $\Stau$.
Corollary \ref{HilbIsoCor} and the results in Section \ref{HigherDimHilb}
hold for all $d$ such that $d/6\equiv 1\mod 3$, using Remark \ref{otherd}.

\subsection{Hilbert schemes of two points}
For a cubic fourfold $X$ we denote by $F(X)$ the Fano variety of lines on $X$,
a four-dimensional hyperkähler variety of $\kthree^{[2]}$ type.
Hassett proved the following:
\begin{theorem}[{\cite[Thm.\ 6.1.4]{HassettPaper}}]
Assume that $d$ satisfies
\[d = 2(n^2+n+1)\]
for some integer $n\geq 2$. Let $X$ be a generic cubic fourfold in $\mC_d$.
Then $F(X)$ is isomorphic to $\hilb^2(S)$, where $(S,L)\in\mM_d$ is associated to $X$.
\end{theorem}

If also $d\equiv 0\mod 6$,
that $F(X)$ is isomorphic to both $\hilb^2(S)$ and $\hilb^2(\Stau)$
(Hassett calls $F(X)$ \emph{ambiguous}).
Since birationality specializes in families of hyperkähler manifolds,
it follows that $\hilb^2(S)$ is birational to $\hilb^2(\Stau)$
for all K3 surfaces $S$ of degree $d$.

\medskip
We can generalize this using the following result by Addington.
See also Remark \ref{AddingtonIsomorphic}.
\begin{theorem}[{\cite[Thm.~2]{AddingtonRatConj}}]
A cubic fourfold $X$ lies in $\mC_d$ for some $d$ satisfying
\[\text{\threestarsp}\htab a^2d = 2(n^2+n+1)\]
if and only if $F(X)$ is birational to $\hilb^2(S)$ for some K3 surface $S$.
\end{theorem}

Note that \threestarsp implies \twostar.

\begin{lemma}
Suppose that $d$ satisfies \threestar.
Then there exists a choice of the rational map $\varphi\colon\mM_d\dashrightarrow\mC_d$
such that if $(S,L)\in\mM_d$ is associated to $X\in\mC_d$ via $\varphi$,
then $\hilb^2(S)$ and $F(X)$ are birational.
\end{lemma}
\begin{proof}
Consider the primitive sublattices
\[\kdperp\oplus T\subset\extktlat\supset \polktlat\oplus\Z\ell_d\oplus U_4,\]
where $T\supset A_2=\langle \lambda_1,\lambda_2\rangle$
is the orthogonal complement of $\kdperp$ in $\extktlat$.
Then $d$ satisfies \twostarsp if and only if $T\cong\Z\ell_d\oplus U_4$.
Addington showed that \threestarsp holds if and only if $\psi\colon T\to\Z\ell_d\oplus U_4$
can be chosen such that $\psi(\lambda_1)=e_4+f_4$.
Extend $\psi$ to an element of $\tO(\extktlat)$ (use Lemma~\ref{extendauto}
and \cite[Thm.~14.2.4]{LecturesOnK3})
and let $\varphi$ be the induced map $\mM_d\dashrightarrow\mC_d$.

Assume that $(S,L)\in\mM_d$ is associated to $X\in\mC_d$ via $\varphi$.
Choose an isomorphism $\HH^2(S,\Z)\cong U_4^{\perp}\subset\extktlat$ sending $L$ to $\ell_d$,
and consider the induced Hodge structure on $\extktlat$.
There are isometries of sub-Hodge structures
\[\HH^2(F(X),\Z)\cong\lambda_1^{\perp}\cong\psi(\lambda_1)^{\perp}
=(e_4+f_4)^{\perp}\cong\HH^2(M_S(1,0,-1),\Z),\]
where the sign in $M_S(1,0,-1)$ appears because we view the Mukai vector as an element of $\mukailat$.
By Markman's birational Torelli theorem for manifolds of $\kthree^{[n]}$ type,
\cite[Cor.~9.9]{Markman},
$F(X)$ is birational to $M_S(1,0,-1)\cong\hilb^2(S)$.
\end{proof}

\begin{corollary}
When $d\equiv 0\mod 6$ satisfies \threestar, then
$\hilb^2(S)\sim_{\bir}\hilb^2(\Stau)$ for any K3 surface $(S,L)\in\mM_d$.
\end{corollary}

The following proposition shows that we have more than just birationality:
if $d$ is such that $\hilb^2(S)\sim_{\bir}\hilb^2(\Stau)$, then for $S$ generic,
$\hilb^2(S)$ and $\hilb^2(\Stau)$ are isomorphic.

\begin{proposition}\label{BirModelHilb}
Let $(S,L)$ be a polarized K3 surface of degree $d$
with $\pic(S)=\Z L$ and $3| d$. Then $\hilb^2(S)$ has only one birational model.
\end{proposition}
\begin{proof}
By \cite[Thm.~5.1]{DebarreMacri}, the walls of the ample cone
of $\hilb^2(S)$ in the interior of the movable cone are given by the hypersurfaces
$x^{\perp}\subset\ns(\hilb^2(S))\otimes\R$ for all ${x\in\ns(\hilb^2(S))}$
of square $-10$ and divisibility two. We will show that there are no such $x$.

There is an isomorphism 
\[\ns(\hilb^2(S))\cong\ns(S)\oplus\Z\delta = \Z L\oplus \Z\delta,\]
where $\delta$ is a $(-2)$-class orthogonal to $ L$ \cite{BeauvilleHilb}.
So any class in $\ns(\hilb^2(S))$ is given by $aL+b\delta$ for some $a,b\in\Z$,
and its square is $a^2d-2b^2$.
Setting this equal to $-10$ gives the Pell equation $b^2-a^2d/2 = 5$
which, after reducing modulo 3, gives $b^2\equiv 2\mod 3$. This is not possible.

It follows that under any birational map $\hilb^2(S)\dashrightarrow Y$
the pullback of an ample class is ample, thus the map is an isomorphism \cite{Fujiki}.
\end{proof}

\begin{remark}\label{AddingtonIsomorphic}
This also implies that when $d$ satisfies \threestarsp and $3|d$,
then for a generic cubic fourfold $X$ of discriminant $d$,
$F(X)$ is actually \emph{isomorphic} to $\hilb^2(S)$ for a K3 surface $S$ associated to $X$.
It would be interesting to find out what happens when $3\centernot|d$.
In that case there can be divisor classes of square $-10$ (for instance, when $d=62$),
but one would have to check the divisibility to find out
whether $\hilb^2(S)$ has more than one birational model.
\end{remark}

It is natural to ask for the exact conditions on $d$ for $\hilb^2(S)$
to be birational to $\hilb^2(\Stau)$, for all $S$ of degree $d$.
It turns out that \threestarsp is too strong.
We use the following results of \cite{MMY}.

\begin{proposition}[{\cite[Prop.~2.1]{MMY}}]\label{IsoModuliOfSheaves}
Let $(S,L)$ be a polarized K3 surface with ${\pic(S)=\Z L}$.
Let ${v = (x,cL,y)}$ be a primitive isotropic Mukai vector
such that $M_S(v)$ is a fine moduli space. Then
$v = (p^2r,pqL,q^2s)$ for some integers $p,r,q,s$ with $\gcd(pr,qs) = 1$
and $(L)^2=2rs$, and there is an isomorphism $M_S(v)\cong M_S(r,L,s)$.
Moreover, $M_S(r,L,s)$ is isomorphic to $M_S(r',L,s')$ if and only if $\{r,s\}=\{r',s'\}$.
\end{proposition}

\begin{proposition}[{\cite[Prop.~2.2]{MMY}}]\label{BirNumbers}
Let $S_1$ and $S_2$ be two derived equivalent K3 surfaces of Picard number one.
Then $\hilb^n(S_1)\sim_{\bir}\hilb^n(S_2)$ if and only if $S_2 \cong M_{S_1}(p^2r,pqL,q^2s)$
for some $p,q$ with $p^2r(n-1)-q^2s=\pm 1$.
Moreover, $\{r,s\}$ is uniquely determined by $S_2$.
\end{proposition}
Note that $p^2r(n-1)-q^2t=\pm 1$ is equivalent to
$\left((1,0,1-n),(p^2r,pqL,q^2t)\right)=\pm 1$.
So when $(p^2r,pqL,q^2s)$ is primitive then $M_{S_1}(p^2r,pqL,q^2s)$ is a fine moduli space,
isomorphic to $M_{S_1}(r,L,s)$ by Proposition~\ref{IsoModuliOfSheaves}.

\medskip
Our description of $\tau$ gave us $\Stau=M_S(3,L,d/6)$, so $r=3$ and $s=d/6$.
Thus, Proposition~\ref{BirNumbers} tells us that $\hilb^2(S)\sim_{\bir}\hilb^2(\Stau)$
if and only if there exist non-zero integers $p,q$ such that $3p^2-(d/6)q^2=\pm 1$.
Note that $3p^2-(d/6)q^2=1$ does not happen in our case: since $d/6\equiv 1\mod 3$,
reducing modulo 3 gives $q^2\equiv 2\mod 3$ which is not possible.

\begin{corollary}\label{HilbIsoCor}
Suppose that $\rho(S)=1$. Then $\hilb^2(S)$ and $\hilb^2(\Stau)$ are birational
if and only if there exists an integral solution to the equation
\[F:3p^2-(d/6)q^2=-1.\]
Equivalently, $\hilb^2(S)$ admits a line bundle of degree 6.
\end{corollary}
\begin{proof}
A class $aL+b\delta$ in $\ns(\hilb^2(S))= \Z L\oplus \Z\delta$ has square $a^2d-2b^2=6$,
in particular $b=3b_0$ for some $b_0$, if and only if $3b_0^2-(d/6)a^2=-1$.
\end{proof}

Condition \threestarsp implies that $F$ is solvable. 
Namely, assume we have $a^2d/2 = n^2+n+1$.
Multiplying with 4 gives $(2a)^2d/2 = (2n+1)^2+3$.
As $d$ is divisible by 3, so is $2n+1$, and we find that
$3(\tfrac{2n+1}3)^2-(2a)^2d/6=-1$.

In fact, \threestarsp is equivalent to the existence of
a solution to $F$ with $p$ odd and $q$ even.
One can show that such a solution always exists when $d/6$ is a prime $m\equiv 3\mod 4$.
The following example shows that there exist $d$ for which $F$ is solvable
but \threestarsp does not hold.

\begin{example}
Let $d=78$, which satisfies \twostarsp but not \threestarsp (see \cite{AddingtonRatConj}).
Equation $F$ holds with $p=2$ and $q=1$.
In particular, $\hilb^2(S)\sim_{\bir}\hilb^2(\Stau)$ for any $S$ of degree $78$.
\end{example}

More interesting is the next example, where $F$ is not solvable.
Because $S$ and $\Stau$ are derived equivalent, so are $\hilb^n(S)$ and $\hilb^n(\Stau)$
for all $n\geq 1$ \cite[Prop.~8]{PloogHilb}.
Therefore, we obtain two derived equivalent
Hilbert schemes of two points on K3 surfaces which are not birational.
The first example of this phenomenon was given in \cite[Ex.~2.5]{MMY}.
Note the similarity between $\Stau$ and the K3 surface $Y$ in \cite[Prop.~1.2]{MMY}.

\begin{example}\label{HilbTwoNotBir}
Consider $d = 6\cdot 73$. This again satisfies \twostarsp but not \threestar.
Note that $F$ holds if and only if $(3p)^2-(d/2) q^2=-3$,
which is equivalent to $x^2-(d/2)y^2=-3$ when 3 divides $d$.
This is a usual Pell type equation and one can easily check
(using e.g.\ \cite[Thm.~4.2.7]{Pell}) that it has no solution for $d = 6\cdot 73$.
So for $S$ generic of degree $6\cdot 73$, $\hilb^2(S)$ is not birational to $\hilb^2(\Stau)$.
\end{example}

\subsection{Higher-dimensional Hilbert schemes}\label{HigherDimHilb}
For $n\geq 2$, Proposition~\ref{BirNumbers} tells us that $\hilb^n(S)$ and $\hilb^n(\Stau)$ are
birational if and only if there is a solution to
\[F_1 : 3p^2(n-1)-(d/6)q^2 = -1\]
or to
\[F_2 : 3p^2-(d/6)q^2(n-1) = -1.\]
We give some examples for low $n$.

\medskip
$n=3$. The lowest $d$ satisfying $(**)$ and $6\mid d$ is $d= 42$.
Equation $F_1$ with $n=3$ reads $6p^2-7q^2=-1$, which is solved by $p=q=1$.
In general, one can show that if $d/6$ is a prime $m\equiv 5,7\mod 8$, then
$\hilb^3(S)\sim_{\bir}\hilb^3(\Stau)$.

\medskip
$n=4$. In this case, only $F_1$ is solvable and reads $(3p)^2-(d/6)q^2=-1$.
This is always solvable when $d/6$ is a prime $m\equiv 1\mod 4$.
Namely, note that when $m>2$ is prime,
$x^2-my^2=-1$ has a solution if and only if $m\equiv 1\mod 4$.
Reducing this modulo 3 gives $x^2-y^2\equiv -1\mod 3$. This implies that $x^2\equiv 0\mod 3$.
Writing $x=3x'$ gives $9(x')^2-(d/6)y^2=-1$, i.e.\ $F_1$ with $n=4$.

\medskip
$n=5$. Equations $F_1$ and $F_2$ are given by $3(2p)^2-(d/6)q^2=-1$
and ${3p^2-(d/6)(2q)^2=-1}$ which are both solutions for $F$.
Since in $3x^2-(d/6)y^2=-1$ one of $x,y$ has to be even,
the existence of a solution for $F$ also implies the existence
of a solution for $F_1$ or $F_2$.
This shows that $\hilb^2(S)\sim_{\bir}\hilb^2(\Stau)$ if and only if
$\hilb^5(S)\sim_{\bir}\hilb^5(\Stau)$.

\bibliography{InvolutionBib}
\bibliographystyle{alpha}

\end{document}